\pdfoutput=1
\documentclass[10pt]{article}
\usepackage[utf8]{inputenc}
\usepackage[T1]{fontenc}
\usepackage{amsmath,amssymb,amsthm}
\usepackage[a4paper,left=22mm,right=22mm,top=22mm,bottom=24mm]{geometry}
\usepackage{booktabs}
\usepackage{enumitem}
\usepackage{xcolor}
\usepackage[colorlinks=true,linkcolor=blue!55!black,citecolor=red!55!black,
  urlcolor=blue!55!black]{hyperref}

\theoremstyle{plain}
\newtheorem{theorem}{Theorem}[section]
\newtheorem{proposition}[theorem]{Proposition}
\newtheorem{lemma}[theorem]{Lemma}
\newtheorem{corollary}[theorem]{Corollary}
\theoremstyle{remark}

\newcommand{\R}{\mathbb{R}}
\newcommand{\Z}{\mathbb{Z}}
\newcommand{\Q}{\mathbb{Q}}
\newcommand{\CC}{\mathbb{C}}
\newcommand{\E}{\mathbb{E}}
\newcommand{\diam}{\operatorname{diam}}
\newcommand{\dist}{\operatorname{dist}}
\newcommand{\vol}{\operatorname{vol}}
\newcommand{\Aut}{\operatorname{Aut}}

\setlist{itemsep=1pt,topsep=2pt}

\title{\bfseries New upper bounds for the chromatic numbers of Euclidean
spaces:\\[3pt]
\large $\chi(\R^4)\le43$,\; $\chi(\R^5)\le132$,\; $\chi(\R^7)\le1029$,\;
$\chi(\R^9)\le7203$,\; $\chi(\R^{10})\le45619$}
\author{Leonid L. Ivanov\thanks{E-mail: \texttt{leo.ivanov@gmail.com}.}
  \and Nadezhda V. Glushkova\thanks{E-mail:
  \texttt{glushkova.nv@phystech.edu}.}}
\date{September 2026}

\begin{document}
\maketitle

\begin{abstract}
\noindent
A coloring of $\R^n$ is \emph{proper for the forbidden distance segment}
$[1,\ell]$ if no two points of the same color are at a distance from
$[1,\ell]$; the minimum number of colors is $\chi(\R^n,[1,\ell])$, and
$\ell=1$ gives the classical chromatic number $\chi(\R^n)$ of the
Nelson--Hadwiger problem. We prove the new upper bounds
$\chi(\R^4)\le43$, $\chi(\R^5)\le132$, $\chi(\R^7)\le1029$,
$\chi(\R^9)\le7203$, $\chi(\R^{10})\le45619$, improving the previously
known $49$, $140$, $1372$, $17253$ and $3^{10}$; in particular, this refutes
the conjecture of Arman, Bondarenko, Prymak and Radchenko that $49$ and
$140$ are optimal among all lattice colorings of $\R^4$ and $\R^5$. The
first four bounds come from explicit rational lattices --- an Eisenstein
lattice in $\R^4$, a lattice in general position in $\R^5$, and laminations
of the Eisenstein colorings $E_6^*/343$ and $E_8/2401$ in $\R^7$ and $\R^9$
--- and each is reduced, by one verification protocol, to a finite list of
inequalities between explicitly written rational numbers checked in exact
arithmetic. The fifth bound is analytic: we prove that for every Eisenstein
lattice $\Lambda$ the distance between same-colored cells of
$(3+\omega)\Lambda$ equals $\sqrt{7/3}\,\lambda_1(\Lambda)$, which gives the
exact widths of all known colorings with $7^{n/2}$ colors, and a product
rule $\sum_i1/d_i^2\le1$ for the widths of orthogonal products; together they
yield $45619=2401\cdot19$, the first bound in $\R^{10}$ below $3^n$, as well
as $\chi(\R^{25})\le4\cdot7^{12}$ and $\chi(\R^{26})\le19\cdot7^{12}$. We
also show that no sublattice of $E_8$ of index below $2401$ defines a proper
coloring. All code, exact certificates and data are open.

\smallskip
\noindent\emph{MSC 2020:} 52C10 (primary); 05C15, 52C07, 11H31.\\
\emph{Keywords:} chromatic number of Euclidean space, Nelson--Hadwiger
problem, forbidden distance interval, lattice colorings, Voronoi cell,
covering radius, Eisenstein lattices, $E_8$.
\end{abstract}

\section{Introduction}\label{sec:intro}

The \emph{chromatic number} $\chi(\R^n)$ is the minimum number of colors
needed to color the points of $\R^n$ so that no two points at distance
exactly $1$ receive the same color (the Nelson--Hadwiger problem; for the
plane only $5\le\chi(\R^2)\le7$ is known \cite{deGrey,Soifer}). For
$\ell\ge1$ let $\chi(\R^n,[1,\ell])$ be the minimum number of colors of a
coloring in which no two points of the same color realize any distance from
the segment $[1,\ell]$ \cite{Ivanov2011}; forbidding more distances cannot
require fewer colors, so
\begin{equation}\label{eq:mono}
\chi(\R^n)=\chi(\R^n,[1,1])\ \le\ \chi(\R^n,[1,\ell])\qquad(\ell\ge1),
\end{equation}
and every upper bound for the interval version bounds $\chi(\R^n)$. For
the history of the problem see the surveys~\cite{BogolubskyRaigorodskii,Raigorodskii2001,Raigorodskii2013,Raigorodskii2014}.

All known upper bounds in small dimensions come from \emph{lattice
colorings}: the space is cut into the Voronoi cells of a lattice $\Lambda$,
and the cell of $v\in\Lambda$ receives the color of the coset $v+\Gamma$ of
a sublattice $\Gamma\subset\Lambda$ of index $k$. The coloring is proper for
$[1,\ell]$ whenever its \emph{width} --- the ratio of the smallest distance
between same-colored cells to the diameter of a cell --- exceeds $\ell$
(Section~\ref{sec:method}). The best previous values in dimensions $4$--$10$
are due to Arman, Bondarenko, Prymak and Radchenko \cite{ABPR} (ABPR below):
$49$, $140$, $343$, $1372$, $2401$, $17253$ and $3^{10}$. In dimension $4$
the bound $49$ (Coulson; proved, with the minimality of the index $49$ on
the lattice $D_4$, in \cite{ABPR}) improved $54$ of \cite{RadoicicToth},
while the lower bound is $9$ \cite{ExooIsmailescu} (after $7$
\cite{Ivanov2006}). ABPR conjecture (Open question~1) that $49$ and $140$ are
optimal among all lattice colorings of $\R^4$ and $\R^5$. This paper refutes
that conjecture.

\begin{theorem}[main theorem]\label{thm:main}
\[
\chi(\R^4)\le43,\qquad \chi(\R^5)\le132,\qquad \chi(\R^7)\le1029,\qquad
\chi(\R^9)\le7203,\qquad \chi(\R^{10})\le45619 .
\]
More precisely, $\chi(\R^n,[1,\ell])\le k$ for all $\ell\le\ell_0$ for the
triples $(n,k,\ell_0)$ of Table~\ref{tab:main}.
\end{theorem}

\begin{table}[htbp]\centering\small
\caption{The five proven bounds; $\ell_0$ is the proven width of the
forbidden segment ($<x$: proven for every $\ell<x$). Previous bounds are
from \cite{ABPR}.}
\label{tab:main}
\begin{tabular}{@{}c r r l l l@{}}
\toprule
$n$ & previous & new & $\ell_0$ & proof & where \\
\midrule
$4$  & $49$    & $\mathbf{43}$    & $1.00411$        & exact rational certificate & Theorem~\ref{thm:43}\\
$5$  & $140$   & $\mathbf{132}$   & $101/100$        & exact rational certificate & Theorem~\ref{thm:132}\\
$7$  & $1372$  & $\mathbf{1029}$  & $103/100$        & exact rational certificate & Theorem~\ref{thm:1029}\\
$9$  & $17253$ & $\mathbf{7203}$  & $<1.0166127$     & exact rational certificate & Theorem~\ref{thm:7203}\\
$10$ & $3^{10}$& $\mathbf{45619}$ & $<\sqrt{217/214}$& analytic (product rule) & Theorem~\ref{thm:45619}\\
\bottomrule
\end{tabular}
\end{table}

For $n=4,5,7,9$ we exhibit a rational Gram matrix and an integer sublattice
and reduce the inequality $d>\ell_0$, by one protocol
(Proposition~\ref{prop:protocol}), to finitely many inequalities between
rational numbers verified in exact arithmetic; floating point only guided
the search. For $n=10$ the proof is analytic: the \emph{Eisenstein identity}
of Section~\ref{sec:identity} gives the exact width $\sqrt{7/6}$ of the
block $E_8/2401$, and the \emph{product rule} $\sum_i1/d_i^2\le1$ leaves
room for a planar block with $19$ colors ($6/7+4/31<1$). The same rule gives
$4\cdot7^{12}$ and $19\cdot7^{12}$ colors in $\R^{25}$ and $\R^{26}$.
Section~\ref{sec:floor} shows that no sublattice of $E_8$ of index below
$2401$ is admissible. All code, exact certificates and data are open at
\url{https://github.com/ivaleo/Chromatic} (MIT license), together with the
detailed Russian manuscript (search algorithms, ladders of widths, negative
screens, and numerical candidates such as an index-$28812$ coloring of
$\R^{10}$, which we do \emph{not} claim as a bound because its diameter
certificate is numerical).

\section{Width and the certificate protocol}\label{sec:method}

Let $\Lambda\subset\R^n$ be a lattice, $\Gamma\subseteq\Lambda$ a sublattice
of index $k=[\Lambda:\Gamma]$, and $V_0=\{x:\|x\|\le\|x-b\|\ \forall
b\in\Lambda\}$ the Voronoi cell of the origin, a centrally symmetric convex
polytope with $\diam V_0=2R$, $R$ the covering radius. Color the cell $v+V_0$
by the coset $v+\Gamma$. Two points of the same color lie either in one cell
or in cells whose centers differ by $v\in\Gamma\setminus\{0\}$. Put
\begin{equation}\label{eq:defs}
D(v)=\dist\bigl(V_0,\,v+V_0\bigr),\qquad
D(\Gamma)=\min_{v\in\Gamma\setminus\{0\}}D(v),\qquad
d(\Lambda,\Gamma)=\frac{D(\Gamma)}{\diam V_0};
\end{equation}
the open interval $(\diam V_0,D(\Gamma))$ is then free of same-colored
distances.

\begin{proposition}\label{prop:crit}
If $d(\Lambda,\Gamma)>1$, then $\chi(\R^n,[1,\ell])\le[\Lambda:\Gamma]$ for
all $\ell<d(\Lambda,\Gamma)$; in particular $\chi(\R^n)\le[\Lambda:\Gamma]$.
\end{proposition}

\begin{proof}
A scale $s>0$ with $s\diam V_0<1$ and $sD(\Gamma)>\ell$ exists iff
$\ell<d$; for $s\Lambda$ the free interval contains $[1,\ell]$. Use
\eqref{eq:mono} for $\ell=1$.
\end{proof}

\begin{lemma}[midpoint lemma; {\cite[Lemma~2]{Ivanov2011}}]\label{lem:mid}
$D(v)=2\dist(v/2,\,V_0)$.
\end{lemma}

\begin{proof}
If $x\in V_0$, $y\in v+V_0$, then $x'=v-y\in V_0$ and
$\|x-y\|=2\|(x+x')/2-v/2\|\ge2\dist(v/2,V_0)$ by convexity; for the point
$x\in V_0$ closest to $v/2$ and $y=v-x$ equality holds.
\end{proof}

Since $V_0$ lies in the ball of radius $R$, $D(v)\ge\|v\|-\diam V_0$, so
$D(v)\ge\ell\diam V_0$ can only fail inside the \emph{window}
\begin{equation}\label{eq:window}
\|v\|<(1+\ell)\diam V_0=2(1+\ell)R ,
\end{equation}
a provably complete search bound.

\begin{proposition}[certificate protocol]\label{prop:protocol}
Let $Q\in\mathrm{Mat}_n(\Q)$ be symmetric with positive leading principal
minors, $\Lambda=\Z^n$ with $\langle x,y\rangle=x^{\mathsf T}Qy$, and
$\Gamma=C\Z^n$ for an integer matrix $C$. Suppose a finite computation in
exact fraction arithmetic has established:
\textup{(i)} $|\det C|=k$;
\textup{(ii)} for each nonzero class of $\Lambda/2\Lambda$ the shortest
representatives are found, and $V_0$ is cut out by
$\langle x,v\rangle\le\frac12\langle v,v\rangle$ over the classes with a
unique pair $\pm v$ (Voronoi's relevant vectors
\cite[Ch.~21]{ConwaySloane});
\textup{(iii)} all vertices of $V_0$ are enumerated, completeness being
proved by solving all $n$-subsets of facets or by exact closure along the
$1$-skeleton, and $R^2=\max\langle x,x\rangle$ over the vertices;
\textup{(iv)} all $v\in\Gamma\setminus\{0\}$ with $\langle v,v\rangle<W$ are
listed, where $W\ge(2(1+\ell_0)R)^2$ is rational;
\textup{(v)} for each such $v$ a point $x^*\in V_0$ and multipliers $\mu\ge0$
of the active facets with $v/2-x^*=\sum_j\mu_jv_j$ are exhibited (the KKT
conditions of the projection), whence $D(v)^2=4\|v/2-x^*\|^2$ exactly;
\textup{(vi)} $D_{\min}^2-\ell_0^2\cdot4R^2>0$, $D_{\min}$ the minimum over
the window.
Then $d(\Lambda,\Gamma)>\ell_0$ and $\chi(\R^n,[1,\ell])\le k$ for all
$\ell\le\ell_0$.
\end{proposition}

\begin{proof}
(ii) is Voronoi's theorem; (iii): the maximum of $\langle x,x\rangle$ over a
polytope is attained at a vertex and $\diam V_0=2R$; (iv) is
\eqref{eq:window}; (v): KKT conditions are sufficient for projection onto a
polytope, and $D(v)=2\dist(v/2,V_0)$ by Lemma~\ref{lem:mid}; (vi) gives
$D(\Gamma)>\ell_0\diam V_0$, and Proposition~\ref{prop:crit} applies.
\end{proof}

Floating point serves only as an oracle proposing vertices and active sets;
every accepted quantity is recomputed over $\Q$, so an oracle error can stop
a verification but cannot make it wrong.

\section{Four explicit constructions}\label{sec:constr}

Each theorem below is proved by the protocol of
Proposition~\ref{prop:protocol}; we give the data and the quantities the
protocol produces. The full lists of vertices, window vectors and KKT
certificates are in the files named in Section~\ref{sec:repro}.

\subsection{Dimension 4: $\chi(\R^4)\le43$}\label{ssec:43}

Let $S\in\mathrm{GL}_4(\Z)$ have order $3$ and characteristic polynomial
$\Phi_3^2$. A lattice with $S\Lambda=\Lambda$ is a $\Z[\omega]$-module of
rank $2$ ($\omega=e^{2\pi i/3}$), an \emph{Eisenstein} lattice, and the
$S$-invariant forms are the real parts of Hermitian forms
$\bigl(\begin{smallmatrix}a&c\\ \bar c&b\end{smallmatrix}\bigr)$; in the
$\Z$-basis $(e_1,\omega e_1,e_2,\omega e_2)$, with $u=\operatorname{Re}c$
and $s=\tfrac{\sqrt3}{2}\operatorname{Im}c$, the Gram matrix is
\begin{equation}\label{eq:eisform}
Q(a,b,u,s)=\begin{pmatrix}
a & -\tfrac a2 & u & -\tfrac u2-s\\[1pt]
-\tfrac a2 & a & -\tfrac u2+s & u\\[1pt]
u & -\tfrac u2+s & b & -\tfrac b2\\[1pt]
-\tfrac u2-s & u & -\tfrac b2 & b
\end{pmatrix}.
\end{equation}
An $S$-invariant sublattice is a $\Z[\omega]$-submodule, so its index is the
norm of an ideal; $43=N(7+\omega)$, and there are exactly $88$ submodules
of index $43$ (against $81\,400$ sublattices). Restricting the search to
this three-parameter family (instead of nine parameters) produced the
construction; in the full space of quaternary forms the search had stalled
below the threshold.

Let $\Lambda_\star=\Z^4$ with the form \eqref{eq:eisform},
\begin{equation}\label{eq:gram43}
a=1,\qquad b=\frac{39261}{25000},\qquad u=\frac{37381}{100000},\qquad
s=\frac{3141}{12500}\qquad(200\,000\,Q\in\mathrm{Mat}_4(\Z)),
\end{equation}
and let $\Gamma_\star$ be generated by the rows of
\begin{equation}\label{eq:hnf43}
T=\begin{pmatrix}1&0&0&41\\0&1&0&14\\0&0&1&37\\0&0&0&43\end{pmatrix}
\qquad(\text{coordinates in the basis of }\Lambda_\star),\qquad\det T=43 .
\end{equation}
Then $\lambda_1(\Lambda_\star)^2=1$ and
\[
\begin{aligned}
\diam^2V_0&=\frac{5374343884337}{2019775849050}=2.660861544\ldots,\\
D(\Gamma_\star)^2&=\frac{7396075258066147}{2756867811550000}=2.682781969\ldots,\\
d(\Lambda_\star,\Gamma_\star)^2
&=\frac{298768283679964998359742207}{296327113258585430253847000}
=1.008238093\ldots,\qquad d=1.004110598\ldots
\end{aligned}
\]

\begin{theorem}\label{thm:43}
$\chi(\R^4,[1,\ell])\le43$ for all $1\le\ell\le\ell_0=1.00411$; in
particular, $\chi(\R^4)\le43$.
\end{theorem}

\begin{proof}
The $15$ nonzero classes of $\Lambda_\star/2\Lambda_\star$ give $30$ facets
(the maximum in $\R^4$); all $\binom{30}{4}=27\,405$ rational $4\times4$
systems yield exactly $120$ vertices ($f$-vector $(120,240,150,30)$,
$\vol V_0=\det\Lambda_\star$) and $\diam^2V_0$ above. The window
\eqref{eq:window} for $\ell_0$ is enumerated, each $\dist(v/2,V_0)$ is
certified by KKT multipliers in fractions, and the minimum of $D^2$ is the
value above, attained on three pairs $\pm v$ forming one orbit of
$\Aut(\Lambda_\star)\cong\Z/6$. Margin:
\[
D(\Gamma_\star)^2-\ell_0^2\,\diam^2V_0
=\frac{3559631641205182621936510913}{1113651004958403329305500000000000}
=3.196\cdot10^{-6}>0 .
\qedhere
\]
\end{proof}

A second, independent certificate, which uses a supporting direction for
each window vector instead of the projection, bounds
\[
D(\Gamma_\star)^2\ \ge\ \frac{869307639027355969}{324032161065600000}
=2.682781969\ldots
\]
from below, and an exhaustive scan of all $81\,400$ sublattices of index
$43$ confirms that \eqref{eq:hnf43} maximizes $d$ on $\Lambda_\star$. The
point \eqref{eq:gram43} is a rationalization (denominator $10^5$) of the
numerical optimum $d=1.00411205721\ldots$ of the family.

\subsection{Dimension 5: $\chi(\R^5)\le132$}\label{ssec:132}

The previous bound $140$ was proved by ABPR to be minimal among sublattices
of $A_5^*$; the construction below changes the parent form. Let
$\Lambda=\Z^5$ with $\langle x,y\rangle=x^{\mathsf T}Qy$,
\[
Q=\frac1{100000}
\begin{pmatrix}
79327&86284&69597&36140&26501\\
86284&254538&216808&101257&102135\\
69597&216808&296568&163667&140527\\
36140&101257&163667&181733&71653\\
26501&102135&140527&71653&148076
\end{pmatrix}
\]
with the leading principal minors
\[
79327,\quad 12746807270,\quad 1422467480626358,\quad
124437401434750889345,\quad 9999935596575703407615413,
\]
and $\Gamma=\ker\varphi$ for the surjective homomorphism
\[
\varphi(x)=88x_1+89x_2+127x_3+57x_4+3x_5\pmod{132},
\]
so $[\Lambda:\Gamma]=132$ (Smith form $(1,1,1,1,132)$; $\varphi$ combines
characters mod $11$, $3$ and $4$).

\begin{theorem}\label{thm:132}
$\chi(\R^5,[1,\ell])\le132$ for all $1\le\ell\le101/100$; in particular,
$\chi(\R^5)\le132$.
\end{theorem}

\begin{proof}
The cell has $62$ facets and $720$ vertices, all solved over $\Q$, with
\[
R^2=\frac{3093570095915736710467707313641}{3999974238630281363046165200000}.
\]
For $\ell=101/100$ the window \eqref{eq:window} is $\|v\|^2<12.4984\ldots$;
a Fincke--Pohst enumeration over an LLL-reduced basis of $\Gamma$ leaves
$38$ nonzero vectors ($19$ pairs) in it. For all of them the projection onto
$V_0$ is certified over $\Q$; the minimum is attained at
$\pm(2,0,1,-3,0)$ (coordinates in the basis of $\Lambda$):
\begin{gather*}
D_{\min}^2=\frac{1634114770527727}{516898614620000},\\
D_{\min}^2-\Bigl(\frac{101}{100}\Bigr)^2 4R^2
=\frac{1450500657237822255902962097875907780124229}
{258447642807960215367421006872956903000000000}>0 ,
\end{gather*}
so $D_{\min}/\diam V_0=1.010897714548927\ldots>1.01$.
\end{proof}

An independent verifier that uses neither Qhull nor a floating-point
enumerator reproduces the same $R^2$, $D_{\min}^2$ and KKT certificates.

\subsection{Dimension 7: $\chi(\R^7)\le1029$}\label{ssec:1029}

\emph{Lamination} builds a coloring of $\R^{n}$ from a coloring
$(\Lambda',\Gamma'=\ker\psi)$ of $\R^{n-1}$: for a shift $c$, a height $t>0$,
a gluing $a\in\Lambda'/\Gamma'$ and a layer modulus $m$,
\[
\Lambda=\bigl\{(x+ic,\;it):\ x\in\Lambda',\ i\in\Z\bigr\},\qquad
\Gamma=\bigl\{(x+ic,\;it):\ \psi(x)+ia=0,\ m\mid i\bigr\},\qquad
[\Lambda:\Gamma]=[\Lambda':\Gamma']\,m .
\]
The squared diameter of the cell grows by at most $t^2$, and the distances
between ``horizontal'' same-colored cells (zero layer coordinate) do not
decrease, so the excess width of the base pays for the layer height. The
base here is the Eisenstein coloring $E_6^*/343$, $\Gamma'=(3+\omega)E_6^*$,
of width $\sqrt{7/6}$ (Corollary~\ref{cor:widths}), and $m=3$. In a
$\Z$-basis whose first six vectors span the base and whose seventh is the
layer vector $(c,t)$, the rational point we verify is
\begin{equation}\label{eq:gram1029}
\begin{gathered}
Q=\begin{pmatrix}\tfrac34M&g\\ g^{\mathsf T}&17897/10^4\end{pmatrix},\qquad
M=\begin{pmatrix}
4&2&0&-3&2&1\\ 2&4&3&0&1&2\\ 0&3&6&3&0&3\\
-3&0&3&6&-3&0\\ 2&1&0&-3&4&2\\ 1&2&3&0&2&4
\end{pmatrix},\\[3pt]
10^4g^{\mathsf T}=(-8632,\,-4197,\,1601,\,6617,\,-5275,\,4434),
\end{gathered}
\end{equation}
where $\tfrac34M$ is a Gram matrix of $E_6^*$ with $\lambda_1^2=3$, and
$\Gamma=C\Z^7$ with $C$ upper triangular, diagonal $(7,1,7,1,7,1,3)$, the
only nonzero off-diagonal entries $C_{12}=C_{34}=C_{56}=-5$, and
$\det C=1029$ (the three blocks
$\bigl(\begin{smallmatrix}7&-5\\0&1\end{smallmatrix}\bigr)$ realize the
ideal of norm $7$ in the three complex coordinates of the base).

\begin{theorem}\label{thm:1029}
For $\Lambda=\Z^7$ with the form \eqref{eq:gram1029} and $\Gamma=C\Z^7$,
\[
D(\Gamma)^2=7,\qquad
\diam^2V_0=\frac{96858928789031597}{14761819462500000},\qquad
d^2=\frac{103332736237500000}{96858928789031597},\quad d=1.032878\ldots,
\]
so $\chi(\R^7,[1,\ell])\le1029$ for all $1\le\ell\le103/100$; in particular,
$\chi(\R^7)\le1029$.
\end{theorem}

\begin{proof}
Voronoi's theorem on $\Lambda/2\Lambda$ (with the complete bound
$\max_{s\in\{0,1\}^7}Q(s,s)$) gives $127$ pairs, i.e.\ $254$ facets, the
maximum in $\R^7$. The $30\,368$ vertices are obtained as exact solutions of
integer systems; completeness is proved by exact closure along the
$1$-skeleton without assuming simplicity ($1\,600$ vertices have nine active
facets), giving $R^2=96858928789031597/59047277850000000$. The safe window
$\|v\|^2\le2(D_*^2+\diam^2V_0)=200191665026531597/7380909731250000$ contains
$74$ vectors, each with an exact KKT certificate. The minimum $7$ is attained
at $27$ pairs $\pm v$ of horizontal vectors, the images of the $54$ minimal
vectors of $E_6^*$ (the floor $\sqrt{7/3}\,\lambda_1$ of
Proposition~\ref{prop:planar}); the nearest layer vector gives
$D^2=478385948533/68336903600\approx7.000404$. Margin for $\ell=103/100$:
$D(\Gamma)^2-\ell^2\diam^2V_0=5750986852163787427/147618194625000000>0$.
\end{proof}

The inequality $\diam^2V_0<7$ was also obtained by a second route sharing no
computation with the proof above: on the pieces of the common refinement of
two shifted Voronoi partitions of the six-dimensional base ($167$ pieces,
$103$ provably empty) the covering radius is bounded by the maximum of an
explicit convex quadratic, which equals
$103310189182571717/59047277850000000=1.7496\ldots<7/4$, all in fractions.

\subsection{Dimension 9: $\chi(\R^9)\le7203$}\label{ssec:7203}

The same lamination over $E_8/2401$, $\Gamma'=(3+\omega)E_8$ (width
$\sqrt{7/6}$), with $m=3$ and $t^2=8264707/6250000$: in a $\Z$-basis whose
first eight vectors span $E_8$ and whose ninth is the layer vector,
\begin{equation}\label{eq:gram7203}
Q=\begin{pmatrix}\tfrac32A&g\\ g^{\mathsf T}&17982/10^4\end{pmatrix},\qquad
A=\begin{pmatrix}
2&1&1&-1&-1&1&1&1\\ 1&2&1&0&-1&1&1&1\\ 1&1&2&0&0&0&0&0\\ -1&0&0&2&0&0&0&-1\\
-1&-1&0&0&2&-1&-1&-1\\ 1&1&0&0&-1&2&1&1\\ 1&1&0&0&-1&1&2&1\\ 1&1&0&-1&-1&1&1&2
\end{pmatrix},
\end{equation}
\[
10^4g^{\mathsf T}=(3329,\,6390,\,4411,\,-611,\,-3579,\,5606,\,4509,\,8550),
\]
\[
C=\begin{pmatrix}
-5&-2&-9&-5&-3&-6&-10&-3&0\\ 4&3&-4&-3&0&0&-3&0&0\\ 0&0&9&5&3&3&8&3&0\\
0&0&-5&-2&0&-3&-4&0&0\\ 0&0&-4&-3&0&-1&-3&-1&0\\ 0&0&4&3&1&3&3&1&0\\
0&0&0&0&0&0&2&-3&0\\ 0&0&0&0&0&0&1&2&0\\ 0&0&0&0&0&0&0&0&3
\end{pmatrix},
\]
where $A$ is a Gram matrix of $E_8$ ($\det A=1$), $\Gamma=C\Z^9$,
$|\det C|=7203=3\cdot7^4$, and the Smith form of $C$ is
$\mathrm{diag}(1,1,1,1,1,7,7,7,21)$. Unlike $\R^7$, the a priori bound
$\diam^2V_0\le6+t^2$ gives only $R^2\le1.8306>7/4$, so the claim rests on the
complete list of vertices of the nine-dimensional cell; it became computable
after an exact unimodular change of basis.

\begin{theorem}\label{thm:7203}
For $\Lambda=\Z^9$ with the form \eqref{eq:gram7203} and $\Gamma=C\Z^9$,
\[
D(\Gamma)^2=7,\qquad
\diam^2V_0=\frac{1119552292542693}{165294140000000},\qquad
d^2=\frac{1157058980000000}{1119552292542693},\quad d=1.0166127\ldots,
\]
so $\chi(\R^9,[1,\ell])\le7203$ for all $1\le\ell<d$; in particular,
$\chi(\R^9)\le7203$.
\end{theorem}

\begin{proof}
Among $4504$ vectors within the bound $\|v\|^2\le6+t^2$ one finds $376$
pairs, i.e.\ $752$ facets (of at most $1022$; an independent enumeration
over $\Lambda/2\Lambda$ without any bound gives the same list, and a missed
facet could only weaken the bound). Exact closure along the $1$-skeleton
yields $1\,654\,230$ vertices ($1\,649\,640$ simple, $4320$ with fifteen and
$270$ with sixteen active facets) and
$R^2=1119552292542693/661176560000000=1.69327281\ldots<7/4$. The safe window
$\|v\|^2\le2(D_*^2+\diam^2V_0)=2276611272542693/82647070000000$ contains
$280$ vectors ($140$ pairs), each with an exact KKT certificate; the minimum
$7$ is attained at the $240$ horizontal vectors ($120$ pairs), the images of
the minimal vectors of $E_8$, and no layer vector violates the threshold.
Margin at $\ell=1$: $D(\Gamma)^2-\diam^2V_0=37506687457307/165294140000000
=0.2269\ldots>0$; at $\ell=101/100$ it is still
$150036863771988707/1652941400000000000>0$.
\end{proof}

The bound $7203$ supersedes $\chi(\R^9)\le9604$ of Corollary~\ref{cor:9604}
in both parameters (fewer colors, wider segment: $1.0166\ldots$ against
$\sqrt{63/61}=1.0162\ldots$).

\section{The Eisenstein identity and the product rule}\label{sec:identity}

A lattice $\Lambda\subset\R^n$ is \emph{Eisenstein} if it carries a
multiplication by $\omega=e^{2\pi i/3}$ acting as an isometry without
nonzero fixed vectors; then $\alpha\Lambda$ has index $|\alpha|^n$ for
$\alpha\in\Z[\omega]$. The ABPR constructions with $7^{n/2}$ colors in
dimensions $4,6,8,24$ are $\Gamma=(3+\omega)\Lambda$ ($|3+\omega|^2=7$) on
$D_4$, $E_6^*$, $E_8$ and the Leech lattice $\Lambda_{24}$ \cite{ABPR}. Let
$H_1\subset\CC$ be the Voronoi cell of $\Z[\omega]$: the regular hexagon with
facets at distance $1/2$ and vertices at distance $1/\sqrt3$ from the origin.

\begin{proposition}[lower bound via projection onto a plane]\label{prop:planar}
Let $\Lambda$ be Eisenstein, $w\in\Lambda\setminus\{0\}$,
$P=\operatorname{span}(w,\omega w)$ and $H=\|w\|H_1$ the Voronoi cell of
$\Z[\omega]w\subset P$. Then $\pi_P(V_0)\subseteq H$ and
$\dist(\alpha w/2,V_0)\ge\|w\|\dist(\alpha/2,H_1)$ for every
$\alpha\in\Z[\omega]$. In particular
$D((3+\omega)\Lambda)\ge\sqrt{7/3}\,\lambda_1(\Lambda)$.
\end{proposition}

\begin{proof}
The six vectors $\pm w,\pm\omega w,\pm(1+\omega)w=\mp\omega^2w$ lie in
$\Lambda$ and have length $\|w\|$; the half-spaces
$\langle x,u\rangle\le\|u\|^2/2$ over them contain $V_0$ and have normals in
$P$, so $\pi_P(V_0)$ lies in the hexagon they cut out in $P$, which is $H$.
The point $\alpha w/2$ lies in $P$ and $\pi_P$ is $1$-Lipschitz. For
$\alpha=3+\omega$, $\|w\|=1$: the point $(5/4,\sqrt3/4)$ lies in the normal
cone of the vertex $(1/2,1/(2\sqrt3))$ of $H_1$, at squared distance
$9/16+1/48=7/12$; apply Lemma~\ref{lem:mid}.
\end{proof}

\begin{lemma}[a point in the cell]\label{lem:cell}
Let $u_1,\dots,u_s$ be shortest vectors of $\Lambda$, $q=\sum_ic_iu_i$ with
$c_i\ge0$, $\sum_ic_i\le1$, and $\langle q,u_j\rangle\le\frac12\lambda_1^2$
for all $j$. Then $q\in V_0$.
\end{lemma}

\begin{proof}
$q\in V_0$ iff $\langle q,r\rangle\le\frac12\|r\|^2$ for all $r\in\Lambda$.
For $r\ne u_i$, $\|r-u_i\|^2\ge\lambda_1^2=\|u_i\|^2$ gives
$\langle r,u_i\rangle\le\frac12\|r\|^2$; hence
$\langle q,r\rangle\le(\sum_ic_i)\frac12\|r\|^2\le\frac12\|r\|^2$ for
$r\notin\{u_i\}$, and for $r=u_j$ it is the hypothesis.
\end{proof}

\begin{proposition}[upper bound]\label{prop:eisup}
Let $\Lambda$ be Eisenstein and $w$ a shortest vector. Then the six vertices
$\pm\omega^kq$, $q=\frac13(2w+\omega w)$, of $H=\|w\|H_1$ lie in $V_0$; hence
$V_0\cap P=H$ and $\dist(\alpha w/2,V_0)=\|w\|\dist(\alpha/2,H_1)$ for all
$\alpha\in\Z[\omega]$.
\end{proposition}

\begin{proof}
From $w+\omega w=-\omega^2w$ we get $\|w+\omega w\|=\|w\|$, i.e.\
$\langle w,\omega w\rangle=-\frac12\lambda_1^2$. Then
$\langle q,w\rangle=\frac13(2-\frac12)\lambda_1^2=\frac12\lambda_1^2$ and
$\langle q,\omega w\rangle=\frac13(-1+1)\lambda_1^2=0$, so $q\in V_0$ by
Lemma~\ref{lem:cell} ($q=\frac23w+\frac13\omega w$); the other vertices
follow by replacing $w$ with $\pm\omega^kw$. By convexity
$H\subseteq V_0\cap P\subseteq\pi_P(V_0)\subseteq H$
(Proposition~\ref{prop:planar}), and the distance from $\alpha w/2\in P$ to
$V_0$ is squeezed between its distances to $H$ from both sides.
\end{proof}

\begin{theorem}[Eisenstein identity]\label{thm:eis}
For an Eisenstein lattice $\Lambda$ and every
$\alpha\in\Z[\omega]\setminus\{0\}$,
\begin{equation}\label{eq:alpha}
D(\alpha\Lambda)=2\,\lambda_1(\Lambda)\cdot\dist\bigl(\alpha/2,\,H_1\bigr),
\end{equation}
and for $|\alpha|>1$ the minimum is attained exactly at $\alpha w$ with $w$
shortest. In particular
\begin{equation}\label{eq:eis}
D\bigl((3+\omega)\Lambda\bigr)^2=\tfrac73\,\lambda_1(\Lambda)^2,\qquad
d\bigl(\Lambda,(3+\omega)\Lambda\bigr)=\frac{\sqrt{7/3}}{\rho(\Lambda)},
\qquad\rho(\Lambda)=\frac{2R}{\lambda_1(\Lambda)},
\end{equation}
and $D(m\Lambda)=(m-1)\lambda_1(\Lambda)$ for every integer $m\ge2$.
\end{theorem}

\begin{proof}
For $v=\alpha w$, Lemma~\ref{lem:mid} and Proposition~\ref{prop:planar} give
$D(v)\ge2\|w\|\dist(\alpha/2,H_1)\ge2\lambda_1\dist(\alpha/2,H_1)$, with
equality for shortest $w$ by Proposition~\ref{prop:eisup}. If $|\alpha|>1$
then $|\alpha/2|\ge\sqrt3/2>1/\sqrt3$, so $\alpha/2\notin H_1$ and the first
inequality is strict for non-shortest $w$. For $\alpha=3+\omega$ the squared
distance is $7/12$; for $\alpha=m$ the closest point of $H_1$ to $m/2$ is
the facet midpoint $1/2$.
\end{proof}

\begin{corollary}[exact widths]\label{cor:widths}
$D_4$, $E_6^*$, $E_8$ and $\Lambda_{24}$ have $\rho=\sqrt2$
\cite{ConwaySloane}, so each ABPR coloring with $7^{n/2}$ colors has width
exactly $\sqrt{7/6}$: $\chi(\R^n,[1,\ell])\le7^{n/2}$ for $\ell<\sqrt{7/6}$,
$n=4,6,8,24$, and $\sqrt{7/6}$ cannot be improved within these
constructions. The Coxeter--Todd lattice $K_{12}$ has $\rho=\sqrt{8/3}$, so
$d(K_{12},(3+\omega)K_{12})=\sqrt{7/8}<1$ --- no $7^6$-coloring of $\R^{12}$
this way (a partial negative answer to Open question~2 of \cite{ABPR}) ---
while $d(K_{12},3K_{12})=2/\rho=\sqrt{3/2}$, so $\chi(\R^{12},[1,\ell])\le3^{12}$
for $\ell<\sqrt{3/2}$.
\end{corollary}

The threshold $\sqrt{7/3}$ appears in \cite[Remark~9]{ABPR} as a sufficient
condition $\rho\le\sqrt{7/3}$ whose users were unknown; \eqref{eq:eis} shows
it is also necessary and gives the exact width of every such lattice, and it
turns the horizontal floors $D_{\min}=\sqrt7$ in Theorems~\ref{thm:1029} and
\ref{thm:7203} into theorems.

\begin{proposition}[product rule]\label{prop:product}
Let $\Lambda=\bigoplus_{i=1}^{s}\Lambda_i$ be an orthogonal direct sum,
$\Gamma=\bigoplus_i\Gamma_i$, $\Gamma_i\subseteq\Lambda_i$,
$d_i=d(\Lambda_i,\Gamma_i)$, $k_i=[\Lambda_i:\Gamma_i]$. The maximum of
$d(\Lambda,\Gamma)$ over independent rescalings of the blocks equals
$\bigl(\sum_id_i^{-2}\bigr)^{-1/2}$, attained when
$\diam^2V_0(\Lambda_i)\propto d_i^{-2}$; hence
$\chi(\R^n,[1,\ell])\le\prod_ik_i$ for all $\ell<(\sum_id_i^{-2})^{-1/2}$,
$n=\sum_in_i$.
\end{proposition}

\begin{proof}
The cell of a direct sum is the product of the cells, so
$\diam^2V_0=\sum_is_i$ with $s_i=\diam^2V_0(\Lambda_i)$, and
$D(v)^2=\sum_iD_i(v_i)^2$ by Lemma~\ref{lem:mid}. Since $D_i(0)=0$, the
minimum over $\Gamma\setminus\{0\}$ is attained at a vector with one nonzero
component, so $d^2=\min_i(d_i^2s_i)/\sum_js_j$. With $t_i=d_i^2s_i$,
$\sum_js_j=\sum_jt_jd_j^{-2}\ge\min_it_i\sum_jd_j^{-2}$, with equality iff
all $t_i$ coincide, i.e.\ $s_i\propto d_i^{-2}$. Apply
Proposition~\ref{prop:crit}.
\end{proof}

Width is thus an expendable resource: a block of width $d_i$ costs
$1/d_i^2$ out of a budget of $1$, and blocks may have different indices.

\begin{theorem}\label{thm:45619}
$\chi(\R^{10},[1,\ell])\le2401\cdot19=45619$ and
$\chi(\R^{26},[1,\ell])\le19\cdot7^{12}$ for all $\ell<\sqrt{217/214}$; in
particular $\chi(\R^{10})\le45619$ and $\chi(\R^{26})\le19\cdot7^{12}$.
\end{theorem}

\begin{proof}
The block $E_8/2401$ has $d_1^2=7/6$ (Corollary~\ref{cor:widths}), cost
$6/7$. The planar block $\Lambda_2=\Z[\omega]$ ($V_0=H_1$, $\diam^2V_0=4/3$),
$\Gamma_2=(5+2\omega)\Z[\omega]$ of index $N(5+2\omega)=19$: by
\eqref{eq:alpha}, $D(\Gamma_2)=2\dist((5+2\omega)/2,H_1)$, and the point
$(2,\sqrt3/2)$ lies in the normal cone of the vertex $(1/2,1/(2\sqrt3))$ at
squared distance $9/4+1/3=31/12$, so $D(\Gamma_2)^2=31/3$, $d_2^2=31/4$, cost
$4/31$. Total $6/7+4/31=214/217<1$. The Leech lattice block
$\Lambda_{24}/7^{12}$ has the same width $\sqrt{7/6}$.
\end{proof}

\begin{corollary}\label{cor:9604}
$\chi(\R^9,[1,\ell])\le2401\cdot4=9604$ and
$\chi(\R^{25},[1,\ell])\le4\cdot7^{12}$ for all $\ell<\sqrt{63/61}$.
\end{corollary}

\begin{proof}
The one-dimensional block $\Z\supset4\Z$ has cells of length $1$ at
same-color distance $3$, so $d=3$ and cost $1/9$; $6/7+1/9=61/63<1$.
\end{proof}

The previous bounds were $3^{10}$, $3^{25}$, $3^{26}$ and $17253$
\cite{ABPR}; $45619$ is the first bound in $\R^{10}$ below the classical
$3^n$ tower. Three colors in the last block fail ($d=2$, $6/7+1/4>1$): the
budget $1/7$ left by $E_8/2401$ requires $d\ge\sqrt7$, which a planar
lattice block first reaches (numerically) at index $19$.

\section{An index floor on $E_8$}\label{sec:floor}

\begin{lemma}[inradius bound]\label{lem:inradius}
$D(v)\le\|v\|-\lambda_1(\Lambda)$ for $\|v\|\ge\lambda_1(\Lambda)$.
\end{lemma}

\begin{proof}
The ball of radius $\lambda_1/2$ lies in $V_0$, so
$\dist(v/2,V_0)\le\|v\|/2-\lambda_1/2$.
\end{proof}

\begin{proposition}[shell floor]\label{prop:shell}
If $D(v)<\diam V_0$ for every $v\in\Lambda\setminus\{0\}$ with
$\|v\|^2<s^*$, then every $\Gamma\subset\Lambda$ defining a proper coloring
satisfies $[\Lambda:\Gamma]\ge(s^*)^{n/2}/(\gamma_n^{n/2}\det\Lambda)$,
$\gamma_n$ the Hermite constant.
\end{proposition}

\begin{proof}
Properness forces $\lambda_1(\Gamma)^2\ge s^*$; Hermite's inequality gives
$\lambda_1(\Gamma)^n\le\gamma_n^{n/2}\det\Gamma=\gamma_n^{n/2}[\Lambda:\Gamma]\det\Lambda$.
\end{proof}

Lemma~\ref{lem:inradius} always allows $s^*=(\diam V_0+\lambda_1)^2$; since
norms run over shells, $s^*$ jumps to the next shell whenever the first shell
above this threshold is forbidden entirely, which is certified by a one-line
witness.

\begin{lemma}[radial witness]\label{lem:radial}
Let $\langle\Lambda,\Lambda\rangle\subseteq\frac1e\Z$, $\|v\|^2=N$, and let
a rational $s\in(0,\frac12)$ satisfy $s\lfloor e\sqrt{NM}\rfloor/e\le M/2$
for every shell $M=\|u\|^2<4s^2N$. Then $sv\in V_0$ and
$D(v)\le(1-2s)\sqrt N$.
\end{lemma}

\begin{proof}
We need $s\langle v,u\rangle\le\|u\|^2/2$ for all $u\in\Lambda$. For
$M\ge4s^2N$ this is Cauchy--Schwarz, $s\sqrt{NM}\le M/2$; for $M<4s^2N$,
$\langle v,u\rangle$ is a multiple of $1/e$ not exceeding $\sqrt{NM}$, and
the hypothesis applies. Then $D(v)\le2\|v/2-sv\|$ by Lemma~\ref{lem:mid}.
\end{proof}

\begin{theorem}[$2401$ is unimprovable on $E_8$]\label{thm:e8}
If $\Gamma\subset E_8$ defines a proper coloring, then $[E_8:\Gamma]\ge2401$,
with equality for $\Gamma=(3+\omega)E_8$.
\end{theorem}

\begin{proof}
For $E_8$: $\lambda_1^2=2$, $R=1$, $\diam V_0=2$, $\det=1$, integral and even.
Lemma~\ref{lem:inradius} forces $\|v\|^2\ge(2+\sqrt2)^2>11$, so the first
candidate shell is $12$. For $\|v\|^2=12$ apply Lemma~\ref{lem:radial} with
$e=1$, $s=1/4$: $4s^2N=3$, only $M=2$ needs a check, and
$\frac14\lfloor\sqrt{24}\rfloor=1=M/2$; hence
$D(v)\le\frac12\sqrt{12}=\sqrt3<2$, the shell $12$ is forbidden, $s^*=14$,
and Proposition~\ref{prop:shell} with $\gamma_8=2$ \cite{Blichfeldt} gives
$[E_8:\Gamma]\ge14^4/2^4=2401$. The sublattice $(3+\omega)E_8$ has index
$2401$ and is proper by Corollary~\ref{cor:widths}.
\end{proof}

Coulson's general bound $\chi_s\ge2^{n+1}-1$ \cite{Coulson,ABPR} gives only
$511$ for $E_8$, so Theorem~\ref{thm:e8} answers Open question~8 of
\cite{ABPR} for this lattice. The same argument (shells $4$, $19/4$, $5$ of
the bcc lattice, witness $e=4$, $s=1/4$) reproves Coulson's floor $15$ on the
bcc lattice from its norm spectrum alone, and gives the floors $305$ on
$E_6^*$ (record $343$) and $5\,688\,009\,064$ on the Leech lattice (record
$7^{12}$; $\gamma_{24}=4$ \cite{CohnKumar}). On lattices in general position
the norm spectrum is almost continuous and the screen is silent; whether
$43$ is minimal in $\R^4$ remains open.

\section{Reproducibility and provenance}\label{sec:repro}

All code, exact certificates and data are at
\url{https://github.com/ivaleo/Chromatic} (MIT license). Run from
\texttt{audit-data/}, the command
\texttt{python -m chromatic\_research.campaigns.verify\_main\_results}
recomputes the final inequalities of all five rows of Table~\ref{tab:main}
from the published fractions in rational arithmetic; with \texttt{--full} it
also runs the complete verifiers of the rows $4$, $5$, $7$ and $9$ (the last
enumerates $1\,654\,230$ vertices and takes hours). The certificates are
\path{results/dim4_k43_verify.json} and
\path{results/cert43_eisenstein.json} (Theorem~\ref{thm:43}),
\path{results/metric_deform_a5_132_refined_certificate.json}
(Theorem~\ref{thm:132}), \path{results/dim7_1029_exact.json}
(Theorem~\ref{thm:1029}) and \path{results/dim9_7203_exact.json}
(Theorem~\ref{thm:7203}); each is self-contained (rational Gram form,
sublattice, exact $R^2$, complete window, KKT certificates) and can be
re-verified by a short independent program. Every claim above is a theorem
or an exact certificate; numerical results without proof strength are
confined to the repository.

\paragraph{Contributions and the use of AI.}
The problem setting, methods and the computational pipeline, lamination, 
the product calculus and the lower bound via projection onto a plane are due to L.\,L.~Ivanov. 
N.\,V.~Glushkova found the index-$43$ lattice, proved Proposition~\ref{prop:eisup}, 
and wrote the exact verifiers of the laminations that turned $1029$ and $7203$ into Theorems~\ref{thm:1029} and~\ref{thm:7203} 
(with the decisive unimodular change of basis in $\R^9$). 
The symmetry-restricted search, the independent re-verifications and the shell floor are joint. 
The computational pipeline, campaign scripts and first drafts of the text 
were produced with systematic use of an AI assistant 
(general-purpose large language models in an agentic coding environment); 
problem statements, search directions, acceptance criteria and final verification belong to the authors, and the theorems do not depend 
on the provenance of the search code. Responsibility for all claims rests with the authors.

\end{document}